\documentclass[11pt]{amsart}
\usepackage{hyperref, color}
\usepackage[english]{babel}
\usepackage{amsfonts, amsmath, amsthm, amssymb,amscd,indentfirst}
\usepackage[percent]{overpic}  
\usepackage{esint}
\usepackage{graphicx}
\usepackage{epstopdf}
\usepackage{amsmath,amssymb,latexsym,indentfirst}
\usepackage[latin1]{inputenc}
\usepackage{comment}

\numberwithin{equation}{section}

\newtheorem{theorem}{Theorem}[section]

\newtheorem{proposition}{Proposition}[section]
\newtheorem{lemma}{Lemma}[section]

\newtheorem{example}{Example}[section]

\DeclareMathOperator{\hess}{Hess}

\def\R{\mathbb{R}}           
\def\S{\mathbb{S}}    

\begin{document}

\title[Gap phenomena for CMC surfaces]
{Gap phenomena for constant mean curvature surfaces}

\author[Barbosa]{Ezequiel Barbosa}
\address{Departamento de Matem\'{a}tica, Universidade Federal de Minas Gerais, 
Belo Horizonte, MG, Brazil}
\email{ezequiel@mat.ufmg.br}

\author[Cavalcante]{Marcos P. Cavalcante} 
\address{ Instituto de Matem\'{a}tica, Universidade Federal de Alagoas,
Macei\'o, AL, Brazil} 
\email{marcos@pos.mat.ufal.br}

\author{Edno Pereira}
\address{Departamento de Matem\'{a}tica e Estat\'{i}stica, Universidade Federal de S\~ao Jo\~ao del Rei, 
 S\~ao Jo\~ao del Rei, MG, Brazil}
\email{ednopereira@ufsj.edu.br}


\date{\today}


\begin{abstract}

In this paper, we prove gap results for constant mean curvature (CMC) surfaces. Firstly, we find a natural inequality for CMC surfaces which imply convexity for distance function. We then show that if $\Sigma$ is a complete, properly embedded CMC surface in the Euclidean space satisfying this inequality, then $\Sigma$ is either a sphere or a right circular cylinder. Next, we show that if $\Sigma$ is a free boundary CMC surface in the Euclidean 3-ball satisfying the same inequality, then either $\Sigma$ is a totally umbilical disk or an annulus of revolution. These results complete the picture about gap theorems for CMC surfaces in the Euclidean 3-space. We also prove similar results in the hyperbolic space and in the upper hemisphere, and in higher dimensions.
\end{abstract}

\date{\today}

\maketitle

\section{Introduction}\label{intro}
Let $\Sigma$ be a two sided $n$-dimensional hypersurface isometrically immersed in a 
Riemannian manifold $M^{n+1}$,
and let denote by $N$ a unit normal vector field globally defined along $\Sigma^n$. The extrinsic geometry
of $\Sigma$ is naturally measured by the first variation of $N$, which is given by its shape operator. 
More precisely, given $p\in \Sigma$, the shape operator of $\Sigma$ with respect to $N$ at $p$ is the 
self-adjoint linear operator given by 
$A:T_p\Sigma\to T_p\Sigma$, $A(X) = -\nabla_XN$, where $\nabla$ stands for the 
Levi-Civita connection on $M$. The eigenvalues of $A$ are the principal curvatures of $\Sigma$ in $M$.

In this setting, some natural curvature functions arise from $A$. The weakest one is 
the mean curvature $H =\textrm{trace}\, A$ and strongest one is $\|A\|^2$, the squared norm of $A$, in the
sense that conditions on $H$ are less restrictive than the similar ones on $\|A\|^2$. For instance,
if $H=0$, $\Sigma$ is called a minimal hypersurface, while if $\|A\|^2=0$, $\Sigma$ is called a totally 
geodesic hypersurface, and  it is well known that in the Euclidean space or in the unit sphere
there are infinitely many distinct complete minimal hypersurfaces,
but only the hyperplanes are totally geodesic hypersurfaces in $\mathbb{R}^{n+1}$, and
only the equators are totally geodesic hypersurfaces in $\mathbb S^{n+1}.$ Moreover, 
it follows from the works of Marques, Neves and Song that on a closed Riemannian manifold we
always have infinitely many distinct compact minimal hypersurfaces, while we may have no totally geodesic ones.

In 1968, J. Simons \cite{Simons} proved that if $\Sigma^n$ is a compact minimal submanifold in the unit sphere $\mathbb{S}^{n+k}$, then

\[
\int_{\Sigma}[(2-1/k)\|A\|^2-n]\|A\|^2d\Sigma \geq0\,.
\]
In particular, if $\|A\|^2\leq \frac{n}{2-1/k}$, then either $\|A\|=0$ or $\|A\|^2=\frac{n}{2-1/k}$. 

It means that there exists a gap in the values of $\|A\|^2$ in the space of all minimal submanifolds in the unit sphere. The submanifolds satisfying $\|A\|^2= n/(2 -1/k)$ were latter classified by the works of Lawson \cite{L}, when $k=1$, and Chern, do Carmo and Kobayshi \cite{CdCK}, for any $k$. They are the Veronese minimal surface in $\mathbb{S}^4$ and the Clifford tori.

Since then, gap phenomena for submanilfolds have been studied by many authors. In particular, the above result was generalized to the class of constant mean curvature (CMC) hypersurfaces in the unit sphere by Alencar and do Carmo in \cite{AdC}. In this case, it is natural to use the traceless second fundamental in order to 
detect such kind of gaps. We recall that the traceless second fundamental form of a submanifold is defined by $\Phi= A -\frac{H}{n}Id$, and $H=trA$ is the 
non-normalized mean  curvature. Clearly, $\|\Phi\|^2\equiv0$ if and only if $\Sigma^n$ is totally umbilical, and in this sense, $\|\Phi\|^2$ measures how much the submanifold deviates from being totally umbilical.

In the case of closed hypersurface in $\mathbb{S}^{n+1}$ we have a family of CMC tori obtained by the product of a sphere $\mathbb{S}^{n-1}_r$ with a circle $\mathbb{S}^{1}_{\lambda}$. More precisely, given $0 < r <1$
and $\lambda=\sqrt{1-r^2}$, the hypersurface $\mathbb{T}_r:=\mathbb{S}^{n-1}_{r}\times \mathbb{S}^{1}_{\lambda}$ has constant mean curvature $H(r)=\frac{n-1-nr^2}{r\sqrt{1-r^2}}$.

If $\Sigma$ is a hypersurface with constant mean curvature $H$, we can choose an orientation
for $\Sigma$ such that $H \geq 0$. In this case we consider the following polynomial
$$
P_H(x)=x^2 + \frac{n-2}{\sqrt{n(n-1)}}Hx - n\Big(\frac{H^2}{n^2} + 1\Big).
$$
Let $C_H:=x_0^2$, where $x_0$
the only positive solution of $P_H(x)=0$
Using these notations Alencar and do Carmo proved that if the traceless second fundamental form satisfies $\|\Phi\|^2\leq C_H$, then:

\begin{enumerate}
\item[i)]  either $\|\Phi\|^2\equiv 0$ and $\Sigma^n$ is totally umbilical in $\S^{n+1}$,

\item[ii)] or $\|\Phi\|^2\equiv C_H$ and $\Sigma^n$ is an $H(r)$-torus in $\S^{n+1}$.
\end{enumerate}

In the context of complete minimal surfaces in the Euclidean space a gap phenomena was discovered by Meeks, P\'erez, and Ros in \cite{MPR}. They proved that if $\Sigma^2$ is a properly embedded minimal surface in $\mathbb{R}^3$ such that
\begin{equation}\label{AN}
 \|A\|^2 \left\langle x,N\right\rangle^2\leq2
\end{equation}
on $\Sigma$, then $\Sigma$ is either a plane or a catenoid centered at the origin 0. In fact, they proved this result assuming $\|A\|^2|x|^2\leq2$, but the proof follows using the weakest inequality (\ref{AN}). Following Meeks, P\'erez, Ros' nomenclature we refer a condition as in (\ref{AN}) as curvature decay condition or curvature pinching condition.

More recently, this problem was considered by Ambrozio and Nunes \cite{AN} in the context of free boundary minimal surfaces in the unit ball, that is, compact minimal surfaces immersed in $\mathbb{B}^3$ whose boundary  is not empty and intersects $\mathbb{S}^2=\partial\mathbb{B}^3$ in a right angle. These surfaces are critical points for the area functional for those variations whose boundaries are free to move on $\mathbb{S}^2$, and they share many similarities with closed minimal surfaces in the unit sphere. In this case, the simplest examples are the equatorial flat disk and a portion of the catenoid, called the {\it critical catenoid}. Such surfaces were detected by the gap theorem of Ambrozio and Nunes. More precisely, if $\Sigma^2$ is a compact free boundary minimal surface in $\mathbb{B}^3$ such that $\|A\|^2\left \langle x,N\right\rangle^2\leq2$, then

\begin{enumerate}
\item[i)] either $\|A\|^2(x) \left\langle x,N \right\rangle^2 \equiv 0$ and 
$\Sigma$ is a flat equatorial disk,
\item[ii)]  or $\|A\|^2(p) \left\langle p,N \right\rangle^2 = 2$ at some point $p$
and $\Sigma$ is a critical catenoid.
\end{enumerate}

We point out that higher dimensional versions of this result where recently obtained by the first and the third named authors with Gon\c calves in \cite{BPG}, and topological versions were established by the second named author, Mendes and Vit\'orio in \cite{CMV}.

In this present paper, we are interested in investigating gap phenomena for constant mean curvature surfaces in the spirit above in order to complete this picture. Our first result reads as follows.

\begin{theorem}\label{teoclaA}
Let $\Sigma \subset \mathbb R^3$ be a complete properly embedded CMC surface.
Assume that for all points $x$ in $\Sigma$,
	\begin{equation}\label{condcomplet1}
		\|\Phi\|^2 \left\langle x,N \right\rangle^2 
		\leq \frac{1}{2}\left(2 + H\left\langle x,N \right\rangle \right)^2\,.
	\end{equation} 
Then, $\Sigma$ is either a plane, or a catenoid or a round sphere centered at the origin, or a  right cylinder having its rotation axis containing the origin.	
\end{theorem}

In fact, catenoids, spheres and cylinders centered at the
origin are precisely the surfaces that contain
points where the equality is attained. Moreover, 
such cylinders and spheres saturates the inequality in all points.
When $H=0$ the result follows from the work of Meeks, P\'erez and Ros.

 Our technique also applies for complete CMC surfaces immersed in the 
 hyperbolic space. If we consider the Poincar\'e ball model $\mathbb H^3 = (\mathbb B^3, \bar g)$, 
 $\bar g = \frac{4}{(1-|x|^2)^2} \left\langle\, \cdot\,,\,\cdot \right\rangle $, 
 we can define the position vector as   $\vec{x}=\sum_{i=1}^{3}x_i\partial_i$ 
 in $\mathbb H^3$, which is a  conformal vector field 
that is, the Lie derivative of $\bar{g}$ with respect to $\vec{x}$ satisfies 
$$
\mathcal{L}_{\vec{x}} \bar{g}  = 2\sigma \bar{g}, 
$$
where $\sigma (x) =\frac{1 + |x|^2}{1-|x|^2}$ is called {\it  potential function.}
Using these notations, we are able to characterize
Delaunay surfaces in $\mathbb H^3$, that is, complete
CMC surfaces of revolution.
 
 \begin{theorem}\label{teohipA}
Let $\Sigma \subset \mathbb H^3$ be a complete properly embedded CMC surface. Assume that for all points $x$ in $\Sigma$,
	\begin{equation}\label{condcomplet2}
		\frac{\|\Phi\|^2}{\sigma^2} \bar g (\vec{x},N )^2 
		\leq \frac{1}{2}\left(2 + \frac{H}{\sigma}\bar g( \vec{x},N )\right)^2\,.
	\end{equation} 
Then, $\Sigma$ is either simply connected or a Delaunay surface in such way that its rotation axis contains the origin.
\end{theorem}

 We will present a unified proof of Theorems \ref{teoclaA} and \ref{teohipA} in Section \ref{proofs}. 
 In our approach we explore the fact that a CMC surface  immersed in a space form is either totally umbilical or its umbilical points are isolated. This fact, together with the  quadratic decay of curvature  condition imply that an appropriate $\Psi$ is convex when restrict to $\Sigma$. Studying the set of critical points of such function, we can prove that  $\Sigma$ is either simply connected or an annulus. In the unified context these conclusions are obtained using a more involved function.

 In the specific case when the ambient space is $\mathbb{R}^3$, we can use  the recent work of Meeks and Tinaglia \cite{MEEKS2018809, meeks2016geometry} 
about embedded CMC surfaces in $\mathbb R^3.$
As consequence of their work we know that round spheres are the only complete simply connected surfaces embedded in $\mathbb{R}^3$ with non-zero constant mean curvature. 
Moreover, if $\Sigma$ is a complete CMC annulus embedded  in $\mathbb{R}^3$, 
then $\Sigma$ is a Delaunay surface, that is, a circular right cylinder or an unduloid. 
Nevertheless, we can check unduloids do not satisfy the gap condition \ref{condcomplet1}. In fact,
we recall that the generatrix curve of a Delaunay surface with parameters $H$ and $B$ 
can be parametrized by $\beta(s) = (x(s), 0, z(s))$, $s\in \mathbb R$, where
\[
x(s)=\frac{1}{H}\sqrt{1+B^2+2B\sin(Hs+\frac{3\pi}{2})} 
\]
and
\[
z(s)=\int_{\frac{3\pi}{2H}}^{s+\frac{3\pi}{2H}}{\frac{1+B\sin(Ht)}{\sqrt{1+B^2+2B\sin(Ht)}}}dt.
\]

Let us assume that $0<B<1$ and $H>0$, which correspond to the unduloids. 
The key observation in this case is that the function  $z$ satisfies $z'(s)>0$ for all $s$. 
Another important observation is that the function $z=z(s)$ goes to infinity when $s$ 
is arbitrary large. 
Using these facts we can construct a sequence of points $p_n=\beta(t_n)$ such that the inequality 
\begin{equation*}
	\|\Phi\|^2 \left\langle x,N\right\rangle^2 > \frac{1}{2}(2+H\left\langle x,N\right\rangle)^2
\end{equation*}
is satisfied at $p_n$. Hence, we obtain that $\Sigma$ is either a sphere, or a right circular cylinder. 

From this observation a new fact  arises here. Unduloids do not satisfy the quadratic decay of 
curvature condition, but there are  some pieces of unduloids inside the balls which 
satisfy the quadratic decay of curvature condition and are free boundary. These facts and the Ambrozio-Nunes' theorem motivated us to investigate gap phenomenon for free boundary CMC surfaces inside the unit ball. This is the subject of our next theorem.

\begin{theorem}\label{teoprin} 
Let $\Sigma$ be a compact free boundary CMC surface in the unit ball $\mathbb B^3$ with $H\neq 0$.
Assume that for all points $x$ in $\Sigma$,
\begin{equation*}\label{eqteoprin}
		\|\Phi\|^2 \left\langle {x},N \right\rangle^2 
		\leq \frac{1}{2}\left(2 + H\left\langle x,N \right\rangle \right)^2\,.
\end{equation*}
Then, 
\begin{enumerate}
\item[i)]either $\|\Phi\|^2 \left\langle {x},N \right\rangle^2 \equiv 0$ 
and $\Sigma$ is a totally umbilical disc;
\item[ii)]or the equality occurs at some point and $\Sigma$ is a part of a Delaunay surface 
having its axis containing the origin. 
\end{enumerate}
\end{theorem}

Figure \ref{calotaanel} below shows the two cases in Theorem \ref{teoprin}.
\begin{figure}[h]
    \centering
    \includegraphics[width=7.5cm]{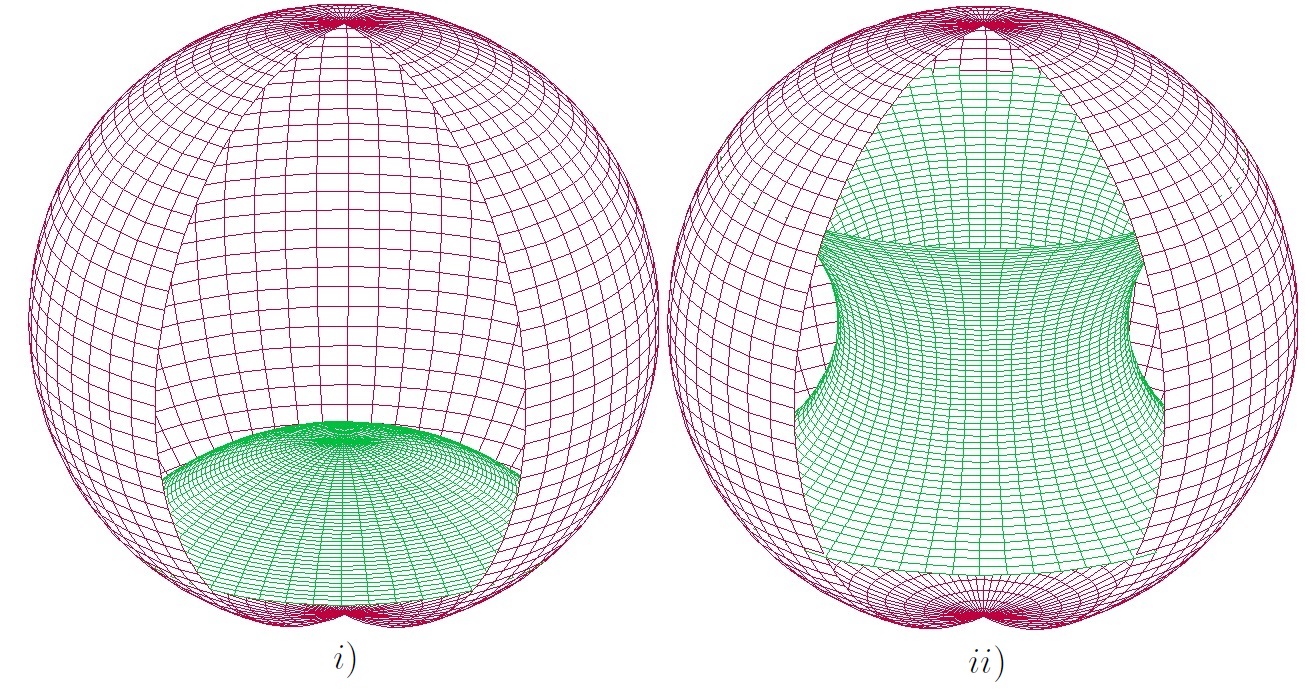}
    \caption{i) spherical cap. ii) CMC annulus.} 
\label{calotaanel}
\end{figure}

Theorem \ref{teoprin} can be obtained in context of  geodesic ball of  space forms $\mathbb{H}^3$ or $\mathbb{S}^3_+$. In this case, we consider the Poincar\'e ball model $\mathbb{H}^3$ for the hyperbolic space as above and for the hemisphere we consider $\mathbb{S}_+^3=(\mathbb{B}_1^3, \bar{g})$, where $\bar g = \frac{4}{(1+|x|^2)^2} \left\langle\, \cdot\,,\,\cdot \right\rangle $. The position vector field $\vec{x}=\sum_{i=1}^{3} x_i\partial_i$ is conformal in $\mathbb{S}^3_+$ and the potential function in this case is given by $\sigma (x) =\frac{1 - |x|^2}{1+|x|^2}$. Let $\textbf{0}$ be the center of $\mathbb{B}_1^3$ and denote by $B_r($\textbf{0}$)$ the geodesic ball in $\mathbb{H}^3$ (or $\mathbb{S}_+^3$) with radius $r$ and centered at $\textbf{0}$. The version of Theorem \ref{teoprin} to space forms can be stated as follows. 

\begin{theorem}\label{teoprin002} 
	Let $B_{r}(\textbf{0})$ be a geodesic ball in hyperbolic space $\mathbb{H}^3$ or hemisphere $\mathbb{S}^3_+$ and let $\Sigma$ be a compact free boundary CMC surface in $B_{r}(\textbf{0})$. Assume that for all points $x$ in $\Sigma$,
	\begin{equation*}\label{condfs}
	\frac{\|\Phi\|^2}{\sigma^2} \bar g (\vec{x},N )^2 
	\leq \frac{1}{2}\left(2 + \frac{H}{\sigma}\bar g( \vec{x},N )\right)^2\,.
	\end{equation*}
	Then, 
	\begin{enumerate}
		\item[i)]either $\|\Phi\|^2 \bar{g}(\vec{x},N)^2 \equiv 0$ 
		and $\Sigma$ is a totally umbilical disc; 
		\item[ii)] or $\Sigma$ rotationally symmetric with a nontrivial topology.  
	\end{enumerate}
\end{theorem}

We point that Theorem 4 in \cite{ABP} includes the ambient spaces considered in Theorems \ref{teoprin} and \ref{teoprin002}, but they have an additional condition in order to get gap results. At this point, in our approach the complex structure of CMC surfaces was essential.

The proofs of Theorem \ref{teoprin} and Theorem \ref{teoprin002} are inspired in the work of Ambrozio and Nunes \cite{AN}, and
follows same steps as in the proof of Theorem \ref{teoclaA}. However, we wish to point out three difficulties circumvented in the case $H\neq 0$. The first one was to found the correct condition which generalizes the condition (\ref{AN}); 
the second one was to guarantee that  condition $(\ref{condcomplet1})$ implies  a good convexity condition for an appropriate function $\Psi$ when restricted to $\Sigma$;  
the third is to show that there are nontrivial examples of surfaces satisfying the conditions in Theorem \ref{teoprin}. As we mention above, in the specific case of free boundary CMC surfaces in the unit ball the natural candidates are portions of some Delaunay surfaces. Although it is intuitive, the computations are involved, since the construction of such surfaces are somewhat ingenious.    

To conclude this Introduction, we would like to mention that we can obtain a sort of gap results in higher dimension for complete minimal hypersurfaces in the Euclidean space, or more generally in Hadamard manifolds by applying the Hardy inequality \cite{Carron}, as we will describe in Section \ref{high}.

\subsection*{Acknowledgments} 
The authors would like to thank the  referee who provided useful and detailed comments.
The first and the third authors were partially supported by  CNPq, CAPES and FAPEMIG/Brazil 
agency grants. The second author was partially supported by CNPq (Grants 309733/2019-7, 201322/2020-0 and 405468/2021- 0) and CAPES (Grants 897/18 and 88887.368700/2019-00). 
Part of this work was done during the second author's stay in Princeton, and he is grateful to Princeton University
 for the support and kind hospitality.

\section{Preliminaries}\label{pre}

Let $\mathbb{B}^3_r$ the Euclidean three-ball with radius $r$ and canonical coordinates $x=(x_1,x_2,x_3)$. Consider in $\mathbb{B}^3_r$ the conformal metric 
of constant curvature $c$, that is,
\begin{equation}\label{metricgc}
\bar{g}=e^{2h}\!\left\langle\, \cdot\,,\,\cdot \right\rangle
\end{equation}
where $h(x)=u(|x|^2)$, and $u$ and $r$ are given by
\begin{equation}\label{funcc}
u(|x|^2)=
\begin{cases}
\ln \left(\dfrac{2}{1-|x|^2}\right),\,\,\, r=1 &\mbox{ for } c=-1,\\
0,\,\,\, r=\infty &\mbox{ for } c=0,\\
\ln \left(\dfrac{2}{1+|x|^2}\right), \,r=1 & \mbox{ for } c=1.
\end{cases} 
\end{equation}
That is, $(\mathbb{B}^3_r,\bar{g})$ is either the hyperbolic space $\mathbb{H}^3$, the euclidean space $\mathbb{R}^3$ or the hemisphere $\mathbb{S}^3_+$,
accordingly to the value of $c$.  

The position vector field $\vec{x}=\sum_{i=1}^{3}x_i\partial_i$ in $(\mathbb{B}_r^{3},\bar{g})$ 
is conformal, 
that is, the Lie derivative of $\bar{g}$ with respect to $\vec{x}$ satisfies 
$$
\mathcal{L}_{\vec{x}} \bar{g} = 2\sigma\bar{g}, 
$$
where $\sigma(x) = 1 + 2u(\left| x \right|^2)\left|x\right|^2$ is called {\it potential function.}

Let $\Sigma$ be a smooth surface in $(\mathbb{B}^3_r,\bar{g})$ and 
let $\varphi: \mathbb{B}^3_r\to J \subset \mathbb R$ the function defined as the squared norm of the position
vector, that is, $\varphi(x)=\bar{g}(\vec{x},\vec{x})$. 
Following the exposition in \cite{ABP} we know that there is 
a differentiable function $\zeta:J \to \R $ such that $\zeta'(s)>0$, $\forall s\in J $ being
the eigenvalues of Hessian of the function $\Psi=\zeta(\varphi)$ restrict to $\Sigma$  given by
\begin{equation}\label{authes}
\bar{\lambda}_1=2\sigma^2\zeta'(s)\left(1+\dfrac{\bar{k}_1}{\sigma}\bar{g}(\vec{x},N)\right)\ \text{and}\  \bar{\lambda}_2
=2\sigma^2\zeta'(s)\left(1+\dfrac{\bar{k}_2}{\sigma}\bar{g}(\vec{x},N)\right),
\end{equation}
where $\bar{k}_1$ and $\bar{k}_2$ are the principal curvatures of $\Sigma$ with respect to the metric $\bar{g}$. 
Note that following the same terminology than \cite{ABP}, we conclude that if $\bar{g}$ is the Euclidean metric, then $\zeta(s)=c_1s+c_2$ for constants $c_1$ and $c_2$ such that $c_1>0$. Thus in the Euclidean case, if we choose  $c_1=1$ and $c_2 = 0$ we have $\Psi(x)=(\zeta \circ \varphi)(x)=|\vec{x}|^2$. 

The next lemmas show that  the gap condition implies the convexity of $\Psi$ restricted to $\Sigma$.
We first prove  the case of free boundary surfaces.

\begin{lemma}\label{propprin1} Under the same conditions of Theorem \ref{teoprin} or Theorem \ref{teoprin002} we have 
 $$
 \hess_{\Sigma}\Psi (x)(Y,Y) \geq 0,
 $$ 
 for all  $x\in \Sigma$ and $Y \in T_x\Sigma$. That is, the Hessian of $\Psi$ restricted to $\Sigma$ is positive semidefinite.
\end{lemma}

\begin{proof}

In order to prove that $\hess_{\Sigma}\Psi (x)(Y,Y) \geq 0$, we need to show that
$\bar{\lambda}_1$ and $\bar{\lambda}_2$ in (\ref{authes}) are nonnegative. 
Arguing  as in \cite{ABP} we easily see that condition (\ref{condcomplet1}) (or (\ref{condcomplet2})) ensures  
$\bar{\lambda}_1$ and $\bar{\lambda}_2$ have  same sign. 
Now, we need to show that at least one $\bar{\lambda}_i$ is non-negative. 
For this end it is enough to show that the function $v$ defined on $\Sigma$ and given by 
$$
v:= \left(1+\dfrac{\bar{k}_1}{\sigma}\bar{g}(\vec{x},N)\right) + \left(1+\dfrac{\bar{k}_2}{\sigma}\bar{g}(\vec{x},N)\right) = 2+\dfrac{H}{\sigma}\bar{g}(\vec{x},N)
$$ 
is nonnegative, since it implies $\bar\lambda_1+\bar\lambda_2\geq 0$.  
In the following, let us assume that $\Sigma$ is not totally umbilical and that $v(p)<0$ at some point $p\in \Sigma$. The free boundary condition ensures that
$$
v =2+\frac{H}{\sigma}\bar{g}(\vec{x},N) = 2 
$$ along $\partial \Sigma$. 
Choose $q \in \partial \Sigma$ and let $\alpha:[0,1] \rightarrow \Sigma$ be a continuous curve such that 
$\alpha(0)=p$ and $\alpha(1)=q$ (see Figure \ref{fig2}). 
Since $v$ changes signal along $\alpha$, there is a point $p_0 =\alpha(t_0)$, $t_0\in(0,1)$ such that $v(p_0)=0$. On the other hand, since $0=v(p_0)=2+\frac{H}{\sigma}\bar{g}(\vec{x},N)$, we have 
$\bar{g}(\vec{x},N)(p_0)\neq 0$. This together with our curvature decay condition implies that $\|\Phi\|^2(p_0)=0$, 
and hence $p_0$ is an umbilical point. Since $(\mathbb{B}^3_r,\bar{g})$ has constant sectional curvature 
and $\Sigma$ is a CMC surfaces that is not a totally umbilical,
$p_0$ is an isolated point.
Let  $\varepsilon > 0$ such that $v(\alpha(t))<0$, if $t\in[t_0-\varepsilon,t_0)$ and $v(\alpha(t))>0$, 
if $t \in (t_0,t_0+\varepsilon]$, or vice-versa. 

Let $\mathbb{D}_{r_0}(p_0)$ be a geodesic disk with radius $r_0$ centered at $p_0$ such that $p_0$ is 
the only umbilical point of $\Sigma$ on $\mathbb{D}_{r_0}(p_0)$. 
We can choose $r_0$ and $\varepsilon$ in such way that $\alpha(t) \in \mathbb{D}_{r_0}(p_0)$ for all 
$t \in [t_0-\varepsilon,t_0+\varepsilon]$. 
Choose $\tilde{r}_0<r_0$ such that $\alpha(t_0-\varepsilon)$, $\alpha(t_0+\varepsilon)$ 
$\notin$ $\mathbb{D}_{\tilde{r}_0}(p_0)$. 
Let $\mathcal{A}=\mathbb{D}_{r_0}(p_0) \setminus \mathbb{D}_{\tilde{r}_0}(p_0)$ be the 
annulus determined by these two disks and let $\beta$ denotes a path in $\mathcal{A}$ joining the 
points $\alpha(t_0-\varepsilon)$ and $\alpha(t_0+\varepsilon)$, see Figure \ref{fig2}. 
Again, $v$ changes the signal along of $\beta$, and therefore there is a point $\tilde{q}\in \mathbb{D}_{r_0}(p_0)$ 
such that $v(\tilde{q})=0$. But, as above, it implies that $\tilde{q}$ is another umbilical point in  $\mathbb{D}_{r_0}(p_0)$
which is a contradiction, and therefore we have that $v\geq 0$ as desired. 
\end{proof}
\begin{figure}[h]\label{fig2}
\begin{overpic}[width=0.6\textwidth,tics=10]{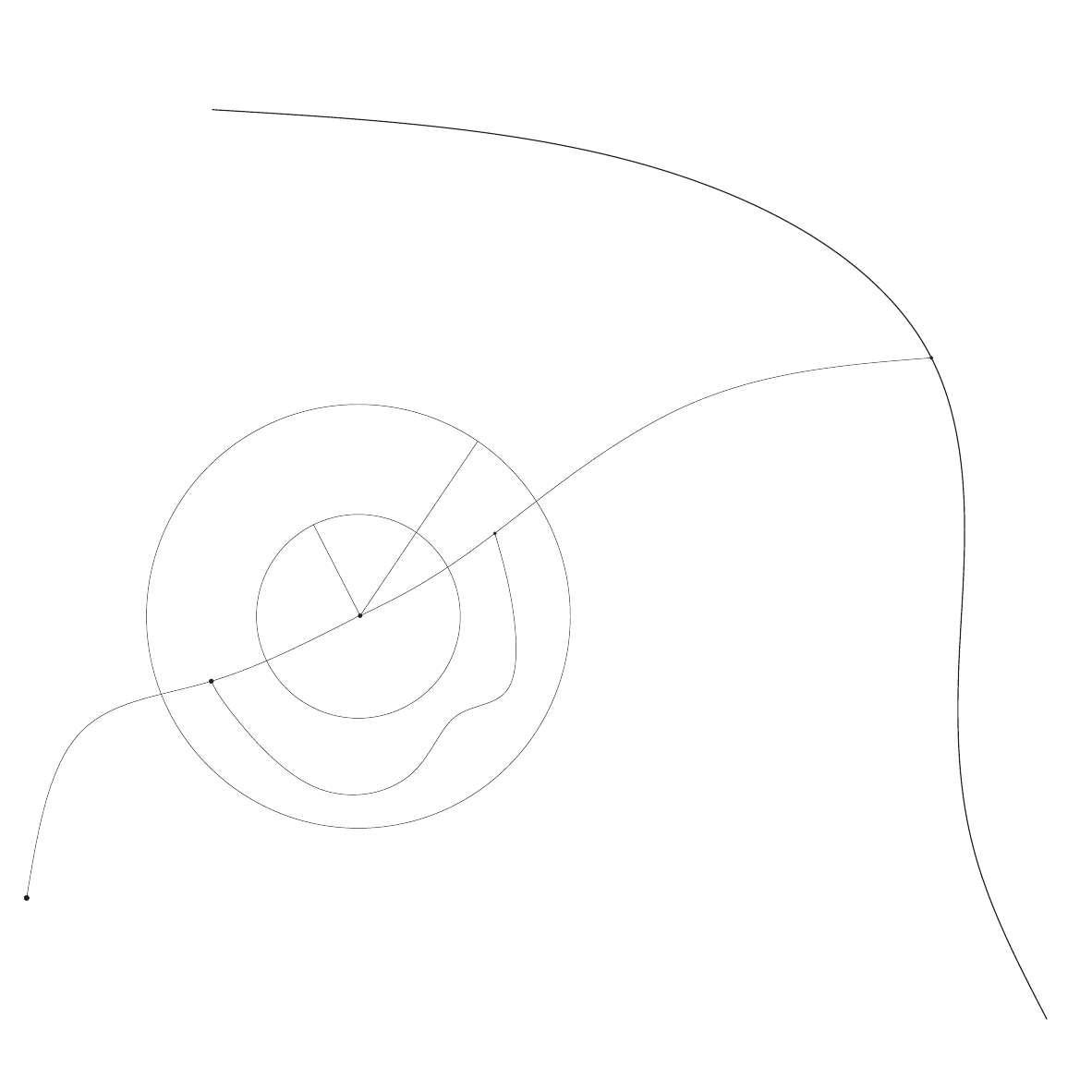}
\put (5,35) {$\displaystyle\alpha$}
\put (42,31) {$\displaystyle\beta$}
\put (87,67) {$\displaystyle q$}
\put (1,13) {$\displaystyle p$}
\put (37,57) {$\displaystyle r_{_0}$}
\put (27,46) {$\displaystyle \tilde{r}_{_0}$}
\put (32,40) {$\displaystyle p_{0}$}
\put (22,65) {$\displaystyle \mathbb{D}_{{r}_{_0}}$}
\put (17,45) {$\displaystyle \mathbb{D}_{\tilde{r}_{_0}}$}
\put (65,85) {\huge $\displaystyle \partial \Sigma $}
\end{overpic} 
\caption{Analysis of the sign of $v$.} 
\end{figure}
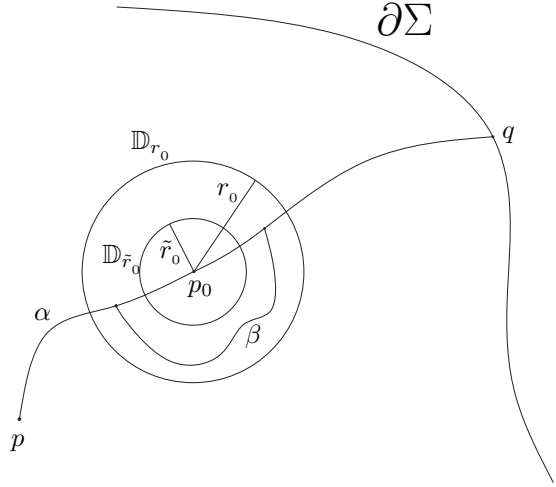

The next lemma deals with the complete case.

\begin{lemma}\label{propprinn} Under the same conditions of Theorem \ref{teoclaA} of Theorem \ref{teohipA} we have 
	$$
	\hess_{\Sigma}\Psi (x)(Y,Y) \geq 0,
	$$ 
	for all  $x\in \Sigma$ and $Y \in T_x\Sigma$. That is, the Hessian of $\Psi$ restrict to $\Sigma$ is positive semidefinite.
\end{lemma}

\begin{proof} We may assume that $\Sigma$ is not umbilical.
As in the previous lemma, it is enough to show that 
$v> 0$ at some point. 
	Let $\textbf{0}$ be the center of $\mathbb{B}^3_r$. If  $\Sigma$ does  pass through $\textbf{0}$ 
	we have immediately $v(\textbf{0})>0$ and we are done. 
	Otherwise we can consider a family of geodesic spheres $\mathbb S^2_\rho\subset\mathbb B_r^3$ 
	centered at $\textbf{0}$ which do not intersect $\Sigma$.
	Let $\tau=\sup\{\rho>0:  \mathbb{S}^2_\rho\cap \Sigma = \emptyset\}$ 
	and consider $q\in \mathbb{S}^2_\tau\cap \Sigma$. 
	
	If  $N_\tau$ denotes the orientation in $\mathbb{S}^2_\tau$ pointing to the center of $\mathbb{S}^2_\tau$
	we choose an orientation $N$ in $\Sigma$ in such way that $N=N_\tau$ at $q\in\Sigma$ and denote by 
	$H_\tau$ the mean curvature of $\mathbb{S}^2_\tau$ with respect to $\bar{g}$. 
	The tangency principle ensures that $H < H_\tau$. 
	Moreover, as $\vec{x}$ is pointing in opposite direction of $N$ at 
	$q \in \Sigma$, we have $\bar{g}(\vec{x},N)=-e^{u(\tau^2)}\tau$. 
	Since 
	$$
	H_\tau=\frac{2}{e^{u(\tau^2)}\tau} \mbox{ and } \sigma \geq 1 \mbox{ for } c\in \{-1,0\},
	$$
	it follows that, 
	\begin{equation*}
	v(q)=2+\dfrac{H}{\sigma}\bar{g}(\vec{x},N) > 2 - \frac{H_\tau}{\sigma}e^{u(\tau^2)}\tau = 2-\frac{2}{\sigma}\geq 0.
	\end{equation*}
\end{proof}

\section{Proofs of the main results}\label{proofs}
\subsection{Proof of Theorems \ref{teoclaA} and \ref{teohipA}} Let us then assume $\Sigma$ is not a totally umbilical and consider $\mathcal{C}=\{p\in\Sigma; \Psi(p)=\min_{_\Sigma}\,\Psi\}$ the set of minimum of the function $\Psi$ restricted to $\Sigma$. By Lemma $\ref{propprinn}$, we have that the function $\Psi_{|_\Sigma}$ is convex and consequently all critical points of $\Psi_{|{_\Sigma}}$ are contained in $\mathcal{C}$.

If $\mathcal{C}$ contains a single point $p$ we conclude that  $\Sigma$ is simply connected from standard Morse theory. So, let us suppose that $\mathcal{C}$ contains more than a one point and $\Sigma$ is not simply connected. Let $[\alpha]\in \pi_1(\Sigma,p)$ be a nontrivial homotopy class based at $p \in \mathcal{C}$. In this class we can find a geodesic loop $\gamma$ also based at $p$, and the convexity of $(\Psi\circ\gamma)(t)$ ensures that $\gamma \subset \mathcal{C}$. We claim that $\gamma$ is a regular curve. If
it is not the case we may have $\gamma'(0)\neq\gamma'(l)$ at $p$. Then we can choose $\varepsilon_0>0$ small and for each $\varepsilon<\varepsilon_0$ consider the minimizing geodesic $\tilde{\gamma}_\varepsilon$ joining $\gamma(l-\varepsilon)$ and $\gamma(0+\varepsilon)$. Again, the convexity of $(\Psi\circ\tilde{\gamma}_\varepsilon)(t)$ implies that $\tilde{\gamma}_\varepsilon\in \mathcal{C}$. Now, we can choose an open set $\mathcal{U}\subset \{\tilde{\gamma}_\varepsilon\}_{\varepsilon<\varepsilon_0}$ such that $\mathcal{U}\subset \Sigma\cap\mathbb{S}^2_{r_0}$, where $\mathbb{S}^2_{r_0}=\{x\in\mathbb{B}^3_r; \Psi(x)=\min_{_\Sigma}\Psi\}$, which is a contradiction because we are assuming that $\Sigma$ is not totally umbilical. Therefore $\gamma$ is regular at $p$, and thus a simple geodesic such that $\gamma([0,l])=\mathcal{C}$. That is, any geodesic loop based at $p$ must be contained in $\mathcal{C}$, and this implies that $\pi_1(\Sigma,p)\approx \mathbb{Z}$ and $\Sigma$ is homeomorphic to an annulus. Moreover, since $\mathcal{C}$ is a simple closed geodesic lying in $\Sigma\cap\mathbb{S}^2_{r_0}$, we conclude that $\mathcal{C}$ is a great circle in $\mathbb{S}^2_{r_0}$. Therefore, $\mathcal{C}$ is a geodesic circle  lying in a totally geodesic plane $\Pi$ passing through the center of $\mathbb{B}^3_r$. Using [\cite{SHMKS}, Lemma 4.2] we conclude that $\Sigma$ is a Delaunay surface meeting $\Pi$ orthogonally along the $\mathcal{C}$. 
Note that, as the rotation axis of $\Sigma$ passes through at the center of the circle  $\mathcal{C}$ and $\mathcal{C}$ is a great circle of a sphere $\mathbb{S}^2_{r_0}$ centered at the origin, we conclude that $\Sigma$ is centered at the origin.

Now, if $\Sigma\subset \mathbb{R}^3$ is simply connected and $H \neq 0$, then it follows by  Meeks-Tinaglia work \cite{meeks2016geometry} that $\Sigma$ is a round sphere. If  $\Sigma\subset \mathbb{R}^3$ is a Delaunay surface and $H \neq 0$ then it must be an unduloid or a right cylinder, because it is embedded. But according to Section \ref{examples}, unduloids does not satisfy the condition \ref{condcomplet1} and therefore, $\Sigma$ is a right cylinder. \begin{flushright}$\qed$  \end{flushright}

As a direct consequence of the proof we have the following more general version of Theorem \ref{teoclaA}.

\begin{theorem} Let $\mathbb{B}^3_r$ be the Euclidean three ball and let  $\bar{g}=e^{2h}\left\langle\cdot,\cdot\right\rangle$ be a conformally Euclidean metric where $h(x)=u(|x|^2)$ and $u:[0,r^2)\rightarrow\mathbb{R}$ is a smooth
function. Let $\sigma$ be the potential function associated with the conformal vector field $\vec{x}$ and suppose that $\sigma(x)\geq1$ for all $x\in \mathbb{B}^3_r$. If $\Sigma\subset (\mathbb{B}^3_r,\bar{g})$ is a complete properly embedded CMC surface whose umbilical points on $\Sigma$ are isolated unless and for all $x\in \Sigma$
$$
\frac{\|\Phi\|^2}{\sigma^2} \bar{g}(\vec{x},N)^2 
\leq \frac{1}{2}\left(2 + \frac{H}{\sigma}\bar{g}(\vec{x},N) \right)^2\,,	
$$
then, $\Sigma$ is either simply connected or homeomorphic to an annulus. 	
\end{theorem}

\subsection{Proof of Theorems \ref{teoprin} and \ref{teoprin002}}\label{proof}

The proof follows the same ideas as in \cite{AN}. As before, the curvature pinching condition implies that if the set  $\mathcal{C}=\{x \in \Sigma; \Psi(x)=\min_\Sigma \Psi\}$ has a single point, then $\pi_1(\Sigma,p)=0$ and $\Sigma$ is a topological disk. By Nitsche theorem \cite{N85} or its extension due to Ros and Souam in \cite{m}, we conclude that $\Sigma$ is a totally umbilical disk and assertion i) follows.

If $\mathcal{C}$ has more than one point then it is a simple closed geodesic $\gamma$ lying in $\Sigma\cap \mathbb{S}^2_{r_0}$. Moreover, $\gamma'$ is a principal direction of $\Sigma$. Let $\pi$ be the plane such that  $\gamma \subset \mathbb{S}^2_{r_0}\cap \pi$ and let $E$ be the unit normal vector to $\pi$. Let $V=\vec{x}\wedge E$ be the vector field induced by rotations of $(\mathbb{B}^3_r,\bar{g})$ around the direction of $E$. It is easy to see that the Lie derivative of the vector field $V$ with respect to $\bar{g}$ satisfies, 
$$
\mathcal{L}_{V}\bar{g}=0,
$$
and therefore $V$ is a Killing vector field with respect to the metric $\bar{g}$.
Define $v:\Sigma \rightarrow \R $ by 
$
v(x)=\bar{g}(V,N).
$
It is well known that $v$ is a solution to the Jacobi equation, that is,
\begin{equation}\label{eqche}
\Delta_\Sigma v + (2c + \|A\|^2) v = 0. 
\end{equation}
We also note that since $N(t)$ is parallel to $\gamma(t)$, 
$v$ vanishes along $\gamma$ and consequently,
\begin{equation*}\label{vnul1}
\displaystyle\dfrac{d}{dt}(v\circ\gamma)=0.
\end{equation*} 
Now, for each $t \in [0,1]$ let $\beta_t:(-\varepsilon,\varepsilon)\rightarrow \Sigma$ 
be a curve such that $\beta_t(0)=\gamma(t)$ and 
$\bar{g}( \frac{d\beta_t}{ds}(0), \gamma'(t)) =0.$
Consider $\tilde{N}(s)$ the restriction of the normal vector field $N$ to $\beta_t$. Then, 
\begin{eqnarray}\label{derv1}
\frac{d}{ds}(v \circ \beta_t)(s)_{|_{s=0}}&=&\frac{d}{ds}(e^{2h})(s)_{|_{s=0}}\langle V,\tilde{N} \rangle_{|_{s=0}}\nonumber\\ &+& e^{2h}\left\{\langle \beta_t'(0)\wedge E,\tilde{N}(0)\rangle + 
\langle \beta_t(0) \wedge E,\tilde{N}'(0)\rangle\right\}. 
\end{eqnarray}

Since $\vec{x}$ and $N$ are collinear along $\gamma(t)$, we have $\langle V,\tilde{N} \rangle_{|_{s=0}}=0$. Also, since $\gamma'(t)\bot\beta'_t(0)$, $\gamma'(t)\bot E$ and $\gamma'(t)\bot \tilde{N}(0)$ we obtain that $\beta'_t(0)$, $E$ and $\tilde{N}(0)$ are in the same plane and therefore, $\langle \beta'_t(0)\wedge E, \tilde{N}(0) \rangle=0$. On the other hand, since $\gamma'(t)$ is a principal direction of $\Sigma$ at $\gamma(t)$, we conclude that $\beta'_t(0)$ is the other principal direction of $\Sigma$ at $\beta_t(0)=\gamma(t)$, and therefore $\tilde{N}'(0)$ and $\beta'_t(0)$ are parallel. Thus,  $\gamma'(t)\bot \tilde{N}'(0)$,  $\gamma'(t)\bot E$, $\gamma'(t)\bot\beta_t(0)$ and again, we have $\langle \beta_t(0)\wedge E, \tilde{N}'(0) \rangle=0$. So, equation (\ref{derv1}) becomes 
$$
\displaystyle\dfrac{d}{ds}(v\circ\beta_t)(s)_{|_{s=0}}=0.
$$

Therefore, $\gamma(t)$ is a critical point of $v:\Sigma \rightarrow \R$, for all $t\in [0,1]$. Since the function $\Psi$ is radially increasing (see \cite{BPG}), we have that the set $\mathcal{C}$ is contained in the interior of $\Sigma$. By  \cite[Theorem 2.5]{ch}, 
the critical points in the nodal set $v^{-1}(0)$ of a non-trivial solution
to equation (\ref{eqche}) 
are isolated. 
Since this is not the case, we conclude that $v \equiv 0$. It means that the Killing vector field 
$V$ is tangent to $\Sigma$. 
Since $V$ is induced by rotations, it is equivalent to say that $\Sigma$ is a rotational surface. 
Note that the rotation axis of $\Sigma$ is orthogonal to $\pi$ and passes through the center of the great circle $\gamma \subset \Sigma \cap \mathbb{S}^2_{r_0}$. As $\mathbb{S}^2_{r_0}$ is centered at the origin, we conclude that  the rotation axis of $\Sigma$ passes through the origin and
it completes the proof.

\section{Construction of free boundary Delaunay surfaces}\label{examples}

In this section, we  show that there are some portions of Delaunay surfaces 
that are free boundary on the unit ball and satisfy  the pinching condition (\ref{condcomplet1}).
Recall that Delaunay surfaces are complete rotational surfaces in $\mathbb R^3$ with
constant mean curvature, and they come in a 2-parameter family $\mathcal D_{H,B}$, where
$H>0$ denotes the mean curvature and $B\geq 0,$ $B\neq 1$. 
If $0<B<1$, Delaunay surfaces are embedded and they are called  {\it unduloids}. 
If $B>1$ they are only immersed and called  {\it nodoids}.
If $B=0$ we get right cylinders and when $B\to 1$ they converge to a string of tangent spheres
with same radii.

In order to produce our examples we need to fix some notations and establish some lemmas. 
Let  $\beta(s) = (x(s), 0, z(s))$
be a smooth  curve parametrized by arc length in the $xz$-plane with $x(s)>0$, and let 
denote by $\Sigma$ the  surface obtained by rotation of $\beta$ around the $z$-axis. 
We start presenting sufficient conditions 
for a general rotational surface  to satisfy the pinching condition  (\ref{condcomplet1}).

\begin{lemma}\label{lecondgeral} 
Suppose that the curve $\beta$ satisfies the following conditions   
\begin{equation}\label{condgeral}
- 1  \leq x''(s)\bigg( x(s) - \frac{x'(s)}{z'(s)}z(s) \bigg), \,\, \mbox{if}\,\, z'(s) \neq 0, 
\end{equation}
\begin{equation}\label{condgeral001}
-1 \leq z(s)z''(s), \,\, \mbox{if}\,\, z'(s)= 0, \,and
\end{equation}
\begin{equation}\label{condgeral0001}
\quad - x(s)x'(s)^2 \leq z'(s)x'(s)z(s). 
\end{equation}
Then, $\Sigma$ satisfies the pinching condition
$$
\|\Phi\|^2 \left\langle x,N\right\rangle^2 \leq \frac{1}{2}(2+H\left\langle x,N\right\rangle)^2
$$
on $\Sigma$.
\end{lemma}

\begin{proof} It suffices to show that,
$$
\lambda_1=1+k_1\left\langle x,N \right\rangle \geq 0 \quad \mbox{and} \quad 
\lambda_2=1+k_2(\left\langle x,N \right\rangle \geq 0,
$$
along $\gamma$.
A straightforward computation shows that $\left\langle x,N \right\rangle =x'(s)z(s) - x(s)z'(s)$ and 
the principal curvatures of $\Sigma$ are given by 
$k_1= x'(s)z''(s) - x''(s)z'(s)$ and $k_2=\frac{z'(s)}{x(s)}$. 

If $z'(s)\neq 0$, we may write $k_1 = -\frac {x''(s)}{z'(s)}$. So we have
$\lambda_1 = 1 + x''(s) \big( x(s) - \frac{x'(s)}{z'(s)}z(s) \big)$ and 
condition (\ref{condgeral}) ensures that $\lambda_1(s)\geq 0$.

If $z'(s) =0$, then $x'(s)^2=1$, $k_1=x'(s)z''(s)$ and therefore 
$\lambda_1(s)=1+x'(s)^2z(s)z''(s)=1+z(s)z''(s)\geq 0$ by condition (\ref{condgeral001}).

Finally,
\begin{eqnarray*}\label{authess2}
\lambda_2(s) 
             &=& \frac{x(s) + z'(s)x'(s)z(s) - x(s)z'(s)^2}{x(s)}\\
						 &=& \frac{z'(s)x'(s)z(s) + x(s)x'(s)^2}{x(s)}	
\end{eqnarray*}
and condition $(\ref{condgeral0001})$ implies that $\lambda_2(s)\geq 0$ as desired. 
\end{proof}

The function $g(s)=x(s) - \frac{x'(s)}{z'(s)}z(s)$ that appears in $(\ref{condgeral})$ 
has an important geometric meaning.
In fact, if $g(s_0)=0,$ then 
\[
\frac{x(s_0)}{z(s_0)}=\frac{x'(s_0)}{z'(s_0)}.
\]
That is,  the directions determined by the position vector and 
velocity vector of $\beta$ at $\beta(s_0)$ are parallel, and thus $\Sigma$ is orthogonal to 
the sphere of radius $R=\|\beta(s_0)\|$. In particular we have the following lemma.

\begin{lemma}\label{fung} Assume that $\beta(s)$ is defined for $s\in[a,b]$ and 
considere $\mathcal{Z}=\{s\in [a,b]; z'(s)=0\}$.
Define the function $g:[a,b]\setminus \mathcal{Z} \rightarrow\R$ by 

\begin{equation}\label{funcg}
g(s):= x(s) - \frac{x'(s)}{z'(s)}z(s).
\end{equation}

Let $s_1<s_2$ be two values in $[a,b]$ such that:  
\begin{enumerate}
\item[i)]  $g(s_1)=g(s_2)=0$,

\item[ii)] $x^2(s_1)+z^2(s_1)=x^2(s_2)+z^2(s_2)=:R^2$ and  

\item[iii)] $x^2(s)+z^2(s)<R^2$ for all $s\in (s_1,s_2).$

\end{enumerate}
Then, the rotation of $\beta_{|_{[s_1,s_2]}}$ produces a free boundary surface inside 
the ball of radius $R$.  
\end{lemma}

We recall that the generatrix curve of a Delaunay surface with parameters $H$ and $B$ 
can be parametrized by 
$\beta(s) = (x(s), 0, z(s))$, where
\[
x(s)=\frac{1}{H}\sqrt{1+B^2+2B\sin(Hs+\frac{3\pi}{2})} 
\]
and
\[
z(s)=\int_{\frac{3\pi}{2H}}^{s+\frac{3\pi}{2H}}{\frac{1+B\sin(Ht)}{\sqrt{1+B^2+2B\sin(Ht)}}}dt.
\]
We point out that theses functions differ from those in \cite{K80} by a translation and a change
of parameters, and we do that in order to have the neck of the surface on the plane $z=0$ at $s=0.$
Let us assume that $0<B<1$. The key observation in this case is that the function 
$z$ satisfies $z'(s)>0$ for all $s$. Let $s_0$ be 
the smaller positive value such that  $x''(s_0)=0$. 
One can easily check that  
$s_0=s_0(H,B)=\dfrac{1}{H}\sin^{-1}(-B) - \dfrac{3\pi}{2H}$, where 
$\sin^{-1}:[-1,0] \rightarrow [\frac{3\pi}{2},2\pi]$. 
Thus, given $s \in \left(-s_0,s_0 \right)$ we have $z'(s)>0$ and $x''(s)>0$. 
We need the following observations.

\begin{lemma}\label{propfuncg}
Fix $0<B<1$, $H>0$, and consider the function $g:[-s_0,s_0] \rightarrow \R $ given by $(\ref{funcg})$. 
Then, 

\begin{enumerate}
\item[i)] $g(0)>0.$ 

\item[ii)]  $g'(0)=g'(s_0)=0$.

\item[iii)]  $g$ is increasing in $(-s_0,0)$ and decreasing in $(0,s_0)$.
\end{enumerate}
\end{lemma}

\begin{proof} Assertion i) follows directly since $g(0) =  \frac{1-B}{H} > 0.$
To proof assertion ii) we observe that, since $\beta$ is parametrized by arc length, then
\begin{equation}\label{glinha}
g'(s) =-\frac{x''(s)}{z'(s)^3}z(s).
\end{equation}
Finally, assertion iii) follows from equation (\ref{glinha}), because
$x''(s)$ and $z'(s)$ are positive in $(-s_0,s_0)$. 
\end{proof}

\begin{proposition}\label{lemafunda}
Fix $0<B<1$, $H>0$,  and set  $z_0=\frac{1-B^2}{HB}$. Then  we have:
\begin{enumerate}
\item[i)] If $z(s_0)<z_0$,  then $g(s)>0$ for all $s\in (0,s_0)$.

\item[ii)] If $z(s_0) \geq z_0$, then $g(\bar{s})=0$ for some $\bar{s} \in (0,s_0]$. 
In particular, the surface obtained by rotation of ${\beta}$ is a free boundary CMC surface in 
$\mathbb B^3_{R_0}$, where $R_0^2=x^2(\bar{s})+z^2(\bar{s})$, and satisfies the pinching condition
(\ref{condcomplet1}).
\end{enumerate}
\end{proposition}

\begin{proof} i) Since $g$ is decreasing in $(0,s_0)$  we have 
$g(s_0)<g(s)$ for all $s \in (0,s_0)$. 
On the other hand,  since $\sin(H s_0+\frac{3\pi}{2})=-B$ and $z(s_0)<z_0$, we have 
\begin{eqnarray*}
g(s_0)&=& x(s_0) -\frac{x'(s_0)}{z'(s_0)}z(s_0) \\ 
      &>& x(s_0) -\frac{x'(s_0)}{z'(s_0)}z_0 \\
			&=& \frac{\sqrt{1-B^2}}{H} -\frac{B\sqrt{1-B^2}}{1-B^2} \frac{(1-B^2)}{BH}=0,
\end{eqnarray*}
and therefore $g(s)>0$.

ii) If $z(s_0)\geq z_0$, then  we get $g(s_0)\leq 0$. 
By assertion i) of Lemma \ref{propfuncg}, $g(0)>0$, and so there is $\bar{s} \in (0,s_0]$ 
such that $g(\bar{s})=0$. 
On the other hand
$x'(-s)=-x'(s)$ and $z'(-s)=z'(s)$, and thus, $g(-\bar{s})=g(\bar{s})=0$. 
Moreover,  $x'(0)=0$ and $x''(s)>0$ imply that $x'(s)>0$ for all $s \in (0,\bar{s}]$. 
Therefore, $x'(s)>0$ and $z'(s)>0$ and it ensures  $x^2(s)+z^2(s)<R_0^2:=x^2(\bar{s})+z^2(\bar{s})$ for
all $s \in (0,\bar{s}]$. 
Because the curve $\beta$ is symmetric with 
respect to $x$-axis we get $x^2(s)+z^2(s)\leq R_0^2$ for all $s \in [-\bar{s},\bar{s}]$ and we
conclude that the surface is free boundary by Lemma \ref{fung}. 
To prove that it satisfies the pinching condition one can easily check that all conditions of Lemma
\ref{lecondgeral} are satisfied.  
\end{proof}
Since the pinching condition (\ref{condcomplet1}) and the free boundary condition are invariant by dilatations of the Euclidean metric, we can construct such examples in the unit ball $\mathbb B^3.$
In the next we present a concrete example.

\begin{example} Choosing $B=0.9$ and $H=0.1$ we have $z_0=\frac{1-B^2}{HB}=2.\bar{1}$ and
$s_0=10\sin^{-1}(-0.9) + 5\pi \approx 4.51026$. Thus we get 
$$
z(s_0)=\displaystyle\int_{15\pi}^{4.51026+15\pi}\left({\frac{1+(0.9)\sin(0.1t)}{\sqrt{1+(0.9)^2+(1.8)\sin(0.1t)}}}\right)dt
\approx 2.71697.
$$
Therefore, $z(s_0)\geq z_0$, and by  Lemma \ref{lemafunda} there is $\bar{s}\in(0,s_0]$ such that the 
portion of  
Delaunay surface corresponding to the revolution of $\beta(s)$ for $s\in [-\bar s, \bar s]$ is a 
CMC annulus in the ball $\mathbb B_{R_0}$, where $R_0=x^2(\bar{s})+z^2(\bar{s})$.   
\end{example}

The next proposition says essentially that there are portions of nodoids that are free 
boundary in the ball and satisfy the conditions of Lemma \ref{lecondgeral}. 
The proof follows the same spirit as in Lemma \ref{propfuncg} and Proposition \ref{lemafunda}.

\begin{proposition} Fix $B>1$ and $H>0$ and
consider $x$, $z$ and $g$ as above and defined in the interval $I_0=(-r_0,r_0)$, 
where $r_0$ the smallest positive value  such that $z'(r_0)=0$. 
Then, there exists $\bar{r} \in (0,r_0)$ such that $g(\bar{r})=g(-\bar{r})=0$. Moreover, 
for all $r\in[-\bar{r},\bar{r}]$ we have $g(r) \geq 0$, $x''(r)>0$ and $x'(r)z(r) \leq 0$.

In particular, the surface obtained by rotation of ${\beta}$ restrict to $I_0$ is a free boundary CMC surface in 
$\mathbb B^3_{R_0}$, where $R_0^2=x^2(\bar{s})+z^2(\bar{s})$, and satisfies the pinching condition
(\ref{condcomplet1}). 
\end{proposition}

To conclude this section we observe that there are points in the unduloids that do not satisfy the pinching condition.
To see this, we define the sequence $p_{n}=\beta(t_n)$
of points on  $\beta$ where
$$
t_n= -\frac{1}{H}\mbox{sin}^{-1}(-B)+\frac{4n\pi+3\pi}{2H}.
$$
That is, the points $t_n$ are defined in order to satisfy 
$\sin(Ht_n + \frac{3 \pi}{2})=-B,$ for all   $n\in \mathbb{N}.$
Moreover,  we have
$$
x(t_n)=\frac{\sqrt{1-B^2}}{H}, 
$$
$$
x'(t_n)=-B,\quad z'(t_n)=\sqrt{1-B^2} 
$$ 
and
$$
x''(t_n)=0.
$$
Another important observation is that the function $z=z(s)$ goes to infinity when $s$ is arbitrary larger. 
Thus, $z(t_n) \rightarrow \infty$ when $n \rightarrow \infty$. 
With these notations we have:

\begin{lemma} Fix $0<B<1$ and $H>0$ and consider  the sequence $p_n=\beta(t_n)$  as above.
Then, there is $n_0 \in \mathbb{N}$ such that the inequality 
\begin{equation}\label{ad001}
\|\Phi\|^2 \left\langle x,N\right\rangle^2 > \frac{1}{2}(2+H\left\langle x,N\right\rangle)^2
\end{equation}
is satisfied for  $p_n,$ for all $n \geq n_0$. 
\end{lemma}

\begin{proof} 
Since $z'(s)\neq 0$ we can write the eigenvalues of $ \hess_{\Sigma}\varphi$ as 
\begin{eqnarray*}
\lambda_1(t_n)&=&1 + x''(t_n) \bigg( x(t_n) - \frac{x'(t_n)}{z'(t_n)}z(t_n) \bigg) = 1.
\end{eqnarray*}
and
\begin{eqnarray*}
\lambda_2(t_{n}) &=& \frac{x(t_n)x'(t_n)^2 + z'(t_n)x'(t_n)z(t_n)}{x(t_n)} = BH\Big(\frac{B}{H} - z(t_n)\Big).		
\end{eqnarray*}

Since $z(t_n)$ $\rightarrow \infty$ when $n \rightarrow \infty$, we can choose $n_0$ such that
$z(t_{n})>\frac{B}{H}$ for all $ n \geq n_0$, and consequently $\lambda_2(t_{n})<0$. 
Thus,  $\lambda_1(t_{n})\lambda_2(t_{n})<0$,  and by the proof of Lemma \ref{propprin1}, 
the inequality (\ref{ad001}) holds at  $p_n$ as desired.
\end{proof}

Once we have a divergent sequence of  points where the pinching condition does not hold, we can find
portions of Delaunay surfaces which are free boundary in a ball with a large radius and that contains some
of those points.

\section{Gaps in higher dimension}\label{high}
As stated in the Introduction, we observe that using the Hardy inequality for submanifolds due to Carron \cite{Carron} a certain gap on the length of the second fundamental form of minimal hypersurfaces implies stability. More precisely, we have.

\begin{proposition} (Carron). Let $\Sigma^n \subset \mathbb{R}^{n+1}$ be a complete minimal hypersurface,
$n\geq3$. Then,
\[
\frac{(n-2)^2}{4}\int_{\Sigma}\frac{u^2}{|x|^2}d\Sigma \leq \int_{\Sigma}\|\nabla u\|^2d\Sigma\,,
\]
for all compactly supported function $u\in C^1_0(\Sigma)$.

In particular, if $\|A\|^2|x|^2\leq(n-2)^2/4$, then $\Sigma^n$ is stable. 
\end{proposition}

We point out that Carron established the Hardy inequality for complete submanifolds immersed in Hadamard manifolds that is, a simply connected Riemannian manifold with nonpositive sectional curvature (see also \cite{BMV} for a generalization).

Once we have that $\Sigma$ is stable, we can apply some known results to obtain gap theorems. As a first example we use the recent work of Chodosh and Li \cite{CLi} to get:

\begin{theorem}
Let $\Sigma^3 \subset \mathbb{R}^4$ be a complete minimal hypersurface such that
\[
\|A\|^2|x|^2\leq\frac{1}{4}\,.
\]
Then $\Sigma^3$ is a hyperplane.
\end{theorem}

Assuming Euclidean volume growth we can use Schoen-Simon-Yau \cite{SSY} and Schoen-Simon \cite{SS} to obtain results in dimensions up to 7. More precisely:

\begin{theorem}
Let $\Sigma^n \subset \mathbb{R}^{n+1}$ be a complete minimal hypersurface, $3\leq n\leq 6$, with Euclidean volume growth and such that
\[
\|A\|^2|x|^2\leq\frac{(n-2)^2}{4}\,.
\]
Then $\Sigma^n$ is a hyperplane.
\end{theorem}

If $n\geq 3$ and $\Sigma$ has finite total curvature we have the following result.

\begin{theorem}
Let $\Sigma^n \subset \mathbb{R}^{n+1}$ be a complete minimal hypersurface, $n\geq 3$,
 such that $\int_{\Sigma}\|A\|^nd\Sigma<+\infty$ and
\begin{equation}\label{Ax}
\|A\|^2|x|^2\leq\frac{n(n-2)}{4}\,.
\end{equation}
Then $\Sigma^n$ is a hyperplane.
\end{theorem}

\begin{proof}
The condition (\ref{Ax}) implies that $\Sigma$ is $\frac{n-2}{n}$-stable (see \cite{ChengZhou} for a definition). 
Using \cite[Theorem 1.1]{BdcS} we have that $\|A\|$ is a  bounded function (in fact, goes to zero at infinity), 
and applying \cite[Theorem 1.1]{ChengZhou} we conclude that $\Sigma$ is either a hyperplane or a 
catenoid. 
However, one can  easily check that the gap condition (\ref{Ax}) is not satisfied by catenoids. 
Thus, $\Sigma$ can only be a hyperplane.  

\end{proof}

\section{Further questions}

In \cite{BV}, the first author and Viana extended Ambrozio-Nunes' Theorem for the higher codimension case. So, it is natural to ask whether versions of Theorems \ref{teoprin} and \ref{teoprin002} hold in higher dimension and higher codimension.

In \cite[Section 8.3]{BPS}, Bettiol, Piccione and Santoro proved that 
Delaunay surfaces provide a smooth 1-parameter family of free boundary annuli in the unit ball. 
It is an interesting question to check if the surfaces in this family satisfy the pinching condition (\ref{condcomplet1}).

Finally, it is interesting to ask what are the best constants in the theorems of Section \ref{high}, and also if they can be extended for complete CMC hypersurfaces in the space forms, or in more general Riemannian manifolds.

\bibliographystyle{amsplain}

\bibliography{BCPbiblio}
\end{document}